\documentclass[12pt, twoside, leqno]{article}

\usepackage{amsmath,amsthm}
\usepackage{amssymb}

\usepackage{enumerate}

\usepackage{graphicx}

\newtheorem{thm}{Theorem}[section]
\newtheorem{cor}[thm]{Corollary}

\newtheorem{prop}[thm]{Proposition}
\newtheorem{prob}[thm]{Problem}

\theoremstyle{definition}
\newtheorem{defin}[thm]{Definition}

\numberwithin{equation}{section}

\newcommand{\set}[1]{\left\{#1\right\}}
\newcommand{\setcon}[2]{\left\{#1\,:\,#2\right\}}
\newcommand{\abs}[1]{\left\lvert#1\right\rvert}

\newcommand{\R}{\mathbb{R}}
\newcommand{\Z}{\mathbb{Z}}
\newcommand{\N}{\mathbb{N}}

\newcommand{\sep}{\textrm{sep}}
\newcommand{\cut}{\textrm{cut}}
\newcommand{\diam}{\textrm{diam}}

\begin{document}


\baselineskip=17pt


\title{A continuum of expanders}

\author{David Hume\\
Facult\'{e} des Sciences d'Orsay\\
Universit\'{e} Paris-Sud\\
F-91405 Orsay\\ 
France\\ 
E-mail: david.hume@math.u-psud.fr
}

\date{}

\maketitle


\renewcommand{\thefootnote}{}



\renewcommand{\thefootnote}{\arabic{footnote}}
\setcounter{footnote}{0}


\begin{abstract}
\noindent A regular equivalence between two graphs $\Gamma,\Gamma'$ is a pair of uniformly proper Lipschitz maps $V\Gamma\to V\Gamma'$ and $V\Gamma'\to V\Gamma$.
Using separation profiles we prove that there are $2^{\aleph_0}$ regular equivalence classes of expander graphs, and of finitely generated groups with a representative which isometrically contains expanders.
\end{abstract}

\section{Introduction}
Over the past forty years the study of expanders has blossomed into a rich theory touching virtually every branch of mathematics. Many different constructions of expanders are now known, the first of which was due to Margulis \cite{Mar75, HLW-exp, Lub-exp1, Lub-exp2}. Despite the intense study, surprisingly little consideration has been given to distinguishing the large-scale geometry of families of expanders until very recently.

In this paper we will think of an $\varepsilon$-expander as an infinite union of distinct finite graphs $X=\bigcup_{n\in\N}\Gamma_n$, with Cheeger constant $h(\Gamma_n)\geq \varepsilon>0$, where each component is equipped with the shortest path metric. A $(d,\varepsilon)$-expander is an $\varepsilon$-expander where every vertex has degree at most $d$.

We distinguish expanders using regular maps. A map $X\to X'$ between graphs is \textit{regular} if it is Lipschitz and pre-images of vertices have uniformly bounded cardinality. This greatly generalises the geometric concepts of isometric, quasi-isometric and coarse embeddings between graphs with bounded geometry. 

We define two graphs $X,Y$ to be \textit{regularly equivalent} if there exist regular maps $VX\to VY$ and $VY\to VX$. To give a few examples, the graph of the Sierpi\'{n}ski triangle is regularly equivalent to a tree, $\Z$ is regularly equivalent to $\N$ while lamplighters over finite groups, $F\wr\Z$, are contained in a single regular equivalence class but countably many quasi-isometry classes \cite{EsFisWhy}.

The main result of the paper is the following.

\begin{thm}\label{thm:manyexp} There exist a family of $2^{\aleph_0}$ $(d,\varepsilon)$-expanders $\setcon{X_r}{r\in\R}$ such that given $r\neq s$ there is no regular map $X_r\to X_s$.
\end{thm}
In fact, we will prove that given any $(d,\varepsilon)$ expander $\bigsqcup_{n\in\N} \Gamma_n$ with \textit{unbounded girth} there exists a set $\mathcal{N}$ of $2^{\aleph_0}$ infinite subsets of $\N$ such that there is no regular map $\bigsqcup_{n\in M} \Gamma_n\to \bigsqcup_{n\in N} \Gamma_n$ whenever $M,N\in\mathcal{N}$ and $M\neq N$.

The girth of a graph $\Gamma$, denoted $g(\Gamma)$, is the length of the shortest positive length simple cycle in $\Gamma$. We say $\bigsqcup_{n\in\N} \Gamma_n$ has unbounded girth if $g(\Gamma_n)\to\infty$ as $n\to\infty$.

\medskip
Mendel and Naor recently produced the first example of two expanders which are not coarsely equivalent \cite[Theorem $9.1$]{MN14} one of bounded girth and another of unbounded girth. The question of distinguishing expanders up to coarse equivalence is also considered by Ostrovskii \cite[Theorem $5.76$]{Os13}.

Two alternative constructions of non coarsely equivalent expanders have appeared since this paper was first written. Das proves that a box space of $SL_n(\Z)$ cannot coarsely embed into a box space of $SL_m(\Z)$ whenever $n>m\geq 3$ \cite{Das}. Khukhro and Valette, meanwhile, prove that there are uncountably many coarse equivalence classes of expanders occuring as box spaces of $SL_3(\Z)$ \cite{KhVal}. This is particularly interesting in combintion with the results in this paper as their expanders do not have unbounded girth.

\medskip
Using small cancellation labellings developed by Osajda \cite{Osaj-14-expander}, certain expanders of unbounded girth can be isometrically embedded into Cayley graphs of finitely generated groups. We prove that this can be done in a way which maintains the above distinctness.

\begin{thm}\label{thm:manygps} There exists a family of finitely generated groups $\setcon{G_r}{r\in\R}$ such that given $r\neq s$ there is no regular map $G_r\to G_s$.
\end{thm}
As a result there is no coarse embedding of $G_r$ into $G_s$ whenever $r\neq s$. \textit{A priori} this greatly generalises previous results which present $2^{\aleph_0}$ quasi-isometry classes of finitely generated groups, see for instance \cite{Grig84}.

The groups we obtain are infinitely presented graphical $C'(1/6)$ small cancellation groups and are therefore acylindrically hyperbolic \cite{GruberSisto,DGO11}, and SQ-universal \cite{Gru15}. Moreover, they do not coarsely embed into any Hilbert space and do not satisfy the Baum--Connes conjecture with coefficients \cite{HigLafSka}.

Coarse embeddings into Hilbert spaces are highly sought, as Yu proved that any group admitting one satisfies the Novikov and coarse Baum-Connes conjectures; two important open questions in topology \cite{Yu00}.

\medskip
Our results are obtained by computing the \textit{separation profile} of graphs, as introduced by Benjamini-Schramm-Tim{\'a}r \cite{BenSchTim-12-separation-graphs}.

Given a finite graph $\Gamma$ with $n$ vertices, the \textit{cut size} of $\Gamma$, $\cut(\Gamma)$, is the minimal cardinality of a set of vertices $S$ such that any connected component of $\Gamma\setminus S$ has at most $n/2$ vertices.

The \textit{separation profile} $\sep_X$ of a graph $X$ evaluated at $n$ is the maximum of the cut sizes of all subgraphs of $X$ with at most $n$ vertices.

We will consider separation profiles up to the natural equivalence $f\preceq g$ if there exists a constant $C$ such that $f(n)\leq Cg(n)+C$ for all $n$.

If $X,Y$ are graphs with uniformly bounded degree and there exists a regular map $\phi:X\to Y$ then $\sep_X\preceq\sep_Y$ \cite{BenSchTim-12-separation-graphs}.

By other names this profile has a longer history, for instance, the Lipton-Tarjan theorem states that the cut size of any $n$ vertex planar graph is $O(\sqrt{n})$ \cite{Lipton-Tarjan}. More recently, Shchur proved that separation occurs naturally as an obstruction to quasi--isometrically embedding a metric space into a tree \cite{Shchur14}.

Our main tool is the result below which states that the separation profile detects the existence of an $\varepsilon$-expander as a subgraph.

\begin{thm}\label{thm:sublinsep} Let $X$ be a graph. Then $sep_X(n)/n\not\to 0$ if and only if $X$ contains an expander as a subgraph.
\end{thm}

As a first step we obtain a control on finite subgraphs which should be of independent interest.

\begin{thm}\label{thm:Chconstsubgraphs} Let $X$ be a graph. The Cheeger constant of any $n$-vertex subgraph of $X$ is at most $4\sep_X(n)/n$.
\end{thm}

Our results also give the first continuum of different separation profiles: the only previously known profiles were: bounded, $\log(n)$, polynomial - $n^{1-1/d}$ for each $d\in\set{2,3,\dots}$ and $n/\log(n)$ \cite{BenSchTim-12-separation-graphs}.

We now adjust our viewpoint and prove a uniform upper bound on separation for graphs with finite Assouad-Nagata dimension. Examples include Cayley graphs of many well-studied classes of groups:  polycyclic groups, lamplighter groups, virtually special groups, mapping class groups and certain relatively hyperbolic groups (see \cite{Hume-MCGTrees} and references therein).

\begin{thm}\label{thm:finAN} Let $X$ be a graph with bounded valency and finite Assouad-Nagata dimension. Then $\sep_X(n)\preceq n/\log(n)$. If $X$ is vertex transitive and has growth at most $Cn^d$, then $\sep_X(n)\preceq n^{(d-1)/d}$.
\end{thm}
This is known to be sharp, since the separation of $\Z^d$ is $n^{(d-1)/d}$ while a direct product of two non-abelian free groups has separation $n/\log(n)$ \cite{BenSchTim-12-separation-graphs}.

From Theorems \ref{thm:Chconstsubgraphs} and \ref{thm:finAN} we therefore deduce that for Cayley graphs of groups with finite Assouad--Nagata dimension any $n$-vertex subgraph has Cheeger constant at most $C/\log(n)$.

\subsection*{Acknowledgements}
The author would like to thank Itai Benjamini for introducing him to the topic and for enlightening conversations, John Mackay for pushing the project in a different direction and Dominik Gruber for informative conversations about groups containing expanders. The author would also like to thank Goulnara Arzhantseva, Romain Tessera, Alain Valette and an anonymous referee for comments on earlier versions of this paper.

\section{Separation and inner expansion}\label{Sect:Exp}

\subsection{Expanders and sublinear separation}

In this section we prove Theorem \ref{thm:sublinsep}. We start by considering finite graphs.

For completeness we recall the definition of the Cheeger constant.

\begin{defin} Let $\Gamma$ be a graph with $\abs{\Gamma}=n$. The vertex-boundary of a subset $A\subseteq V(\Gamma)$ - denoted $\partial A$ - is the set of all vertices in $V(\Gamma)\setminus A$ which are neighbours of some vertex of $A$. The (vertex) \textit{Cheeger constant} of $\Gamma$ is given by
$h(\Gamma)=\min\setcon{\abs{\partial A}/\abs{A}}{\abs{A}\leq n/2}$.
\end{defin}

Note that any graph with at most $n$ vertices admits a cut set of size $\lceil \frac{n}{2}\rceil$.

\begin{prop}\label{prop:upbdCh} Let $\Gamma$ be a graph with $\abs{\Gamma}\geq 2$. Then $\cut(\Gamma)\geq \frac{h(\Gamma)}{4}\abs{\Gamma}$.
\end{prop}
\begin{proof}
Set $\abs{\Gamma}=n$. Let $C$ be any cut set of $\Gamma$ with $\abs{C}\leq \lceil \frac{n}{2}\rceil$, so any connected component of $\Gamma\setminus C$ contains at most $\lfloor \frac{n}{2} \rfloor$ vertices. Define $D$ to be a union of connected components of $\Gamma\setminus C$ containing between $n/4$ and $n/2$ vertices: if no single component exists satisfying these bounds then every component contains at most $\lfloor \frac{n}{4}\rfloor$ vertices and we obtain $D$ by adding components in any order until the desired inequalities are achieved. As $\partial D\subseteq C$ we see that $\abs{C}\geq h(\Gamma)\frac{n}{4}$. Hence, $\cut(\Gamma)\geq n\frac{h(\Gamma)}{4}$.
\end{proof}
This immediately implies Theorem \ref{thm:Chconstsubgraphs}.

For the other bound, we will require a more sensitive type of cut. We use the notation $\Gamma \to_{C} \Gamma'$ to denote that $C$ is a subset of vertices of $\Gamma$ and $\Gamma'$ is some largest connected component of $\Gamma\setminus C$.

\begin{defin} Let $\Gamma$ be a graph with $\abs{\Gamma}\geq 2$. We say $\Gamma \to_{C} \Gamma'$ is a $k$-\textit{efficient cut} if $\abs{\Gamma}-\abs{\Gamma'} > k\abs{C}$.
\end{defin}

From the definition of a cut set it is clear that $\abs{\Gamma}>\abs{\Gamma'}>\frac{\abs{\Gamma}}{2}$ whenever $k\geq c_\Gamma:=\frac{\abs{\Gamma}}{\cut(\Gamma)}$ so there is a unique largest component in this case. As we are working with finite graphs every sequence of $k$-efficient cuts terminates.

\begin{prop}\label{prop:lwbdCh} Suppose $\Gamma$ has $n\geq 2$ vertices and let $\Gamma \to_{C_1} \dots \to_{C_m} \Gamma_m$ be any maximal sequence of $\frac{3c_\Gamma}{2}$-efficient cuts. Then $\abs{\Gamma_m}> \frac{n}{2}$ and $h(\Gamma_m)\geq \frac{\cut(\Gamma)}{2n}$.
\end{prop}

\begin{proof}
Suppose $\abs{\Gamma_m}\leq \frac{n}{2}$. Then $\bigcup C_i$ is a cut set for $\Gamma$ containing at most $\frac{2}{3}\cut(\Gamma)$ points, which is a contradiction.

Now let $A\subset \Gamma_m$ with $\abs{A}\leq \frac{1}{2}\abs{\Gamma_m}$. As $\Gamma_m$ admits no $\frac{3c_ \Gamma}{2}$-efficient cuts, we know that $\abs{A}+\abs{\partial A} \leq \frac{3}{2}c_\Gamma\abs{\partial A}$.

Rearranging this, we see that $\abs{\partial A} \geq \frac{\abs{A}}{2c_{\Gamma}}$, so $\frac{\abs{\partial A}}{\abs{A}} \geq \frac{\cut(\Gamma)}{2n}$.

This holds for all $A$ with $\abs{A}\leq \frac{1}{2}\abs{\Gamma_m}$, so we see that $h(\Gamma_m)\geq \frac{\cut(\Gamma)}{2n}$.
\end{proof}

We now prove Theorem \ref{thm:sublinsep} via the following two propositions.

\begin{prop}\label{prop:upbdsep}  Let $X$ be an infinite graph which contains some $\varepsilon$-expander $\bigcup \Gamma_n$. For every $n$, $\sep_X(\abs{\Gamma_n})\geq \frac{\varepsilon}{4} \abs{\Gamma_n}$.

In particular, $\frac{1}{n}\sep_X(n)\not\to 0$.
\end{prop}
\begin{proof}
This follows immediately from Proposition \ref{prop:upbdCh}.
\end{proof}

\begin{prop}\label{prop:lwbdsep}  Let $X$ be an infinite graph and suppose $\sep_X(n)/n\not\to 0$. Then there exists some $\varepsilon>0$ such that $X$ contains an $\varepsilon$-expander $\bigcup \Gamma_n$.
\end{prop}
\begin{proof}
Let $\varepsilon>0$ be such that $\sep_X(n)\geq 2\varepsilon n$ for all $n$ in some infinite subset $I\subseteq\N$. For each $n\in I$ let $\Gamma'_n$ be a subgraph of $X$ with at most $n$ vertices such that $\cut(\Gamma'_n)\geq 2\varepsilon n$.

By Proposition \ref{prop:lwbdCh} each $\Gamma'_n$ has a subgraph $\Gamma_n$ with $\abs{\Gamma_n}\geq \frac{1}{2} \abs{\Gamma'_n}$ and $h(\Gamma_n)\geq \varepsilon$. Moreover, $\abs{\Gamma_n}\to\infty$ as $n\to\infty$, so $X$ contains a $\varepsilon$-expander.
\end{proof}

Now we concentrate on graphs with uniformly bounded degree, where the separation profile is invariant under regular equivalence. We obtain one immediate consequence of Proposition \ref{prop:lwbdsep}.

\begin{cor}  Let $X$ be a graph of bounded degree which admits a coarse embedding into some $L^p$ space, where $p\in[1,\infty)$. Then $ \sep_X(n)/n\to 0 \ \ \textrm{as} \ \ n\to\infty$.
\end{cor}
Using groups with relative property (T), Arzhantseva-Tessera produce box spaces with sublinear separation - they do not weakly contain expanders - which do not coarsely embed into any uniformly curved Banach space \cite{AT14}.

\subsection{Separation profiles of expanders}
The goal of this section is to prove Theorem \ref{thm:manyexp}.

Let $X=\bigsqcup_{n\in\N} \Gamma_n$ be a $(d,\varepsilon)$-expander where $g(\Gamma_{n+1})>\abs{\Gamma_n}$ for each $n$ and $\abs{\Gamma_{n+1}}/\abs{\Gamma_n}\to\infty$ as $n\to\infty$. As long as $\bigsqcup_{n\in\N} \Gamma_n$ has unbounded girth, this can always be done by passing to a subsequence.

\begin{thm} For any infinite subsets $M,N\subseteq \N$ with $C=M\setminus N$ infinite, there is no regular map from $X(M)=\bigsqcup_{n\in M} \Gamma_n$ to $X(N)=\bigsqcup_{n\in N} \Gamma_n$.
\end{thm}
\begin{proof} Let $c\in C$. By Proposition \ref{prop:lwbdsep}, $\sep_{X(M)}(\abs{\Gamma_c})\geq \frac{\varepsilon}{4}\abs{\Gamma_c}$. Now, let $\Gamma$ be a subgraph of $X(N)$ with $k\leq \abs{\Gamma_c}$ vertices. Either $\cut(\Gamma)=0$ or there is some $d\neq c$ such that $\abs{\Gamma\cap\Gamma_d}\geq \abs{\Gamma}/2$.

If $d<c$ then $\cut(\Gamma)\leq \abs{\Gamma}\leq 2\abs{\Gamma_d} \leq 2\abs{\Gamma_{c-1}}$. However, if $d>c$, then $\Gamma\cap \Gamma_d$ is a forest, so has bounded cut size, hence $\Gamma$ has bounded cut size.
Thus $\sep_{X(N)}(\abs{\Gamma_c})/\abs{\Gamma_c}\to 0$ as $c\to\infty$, so there is no regular map $X(M)\to X(N)$.
\end{proof}

Theorem \ref{thm:manyexp} follows by noticing that there is a family $\mathcal{N}$ of $2^{\aleph_0}$ infinite subsets of $\N$ with $M\setminus N, N\setminus M$ infinite for all distinct $M,N\in\mathcal{N}$.

\subsection{Separation profiles of groups containing expanders}

Osajda's construction of $C'(1/6)$ small cancellation labellings of collections of finite graphs satisfying certain girth and diameter restrictions gives a method for constructing finitely generated groups which isometrically contain a family of expander graphs \cite{Osaj-14-expander}.

We let $\boldsymbol\Gamma=\bigsqcup_{n\in\N} \Gamma_n$ be such a family, which we think of as a $(d,\varepsilon)$ expander whose (oriented) edges are labelled by a finite set $S$. Given any collection of finite graphs $\boldsymbol\Lambda$ with such a labelling, we define a group $G(\boldsymbol\Lambda)$ which is generated by $S$ and satisfies precisely the set of relations obtained by concatenating the labels of edges on simple loops in the graphs.

By \cite{Ol06}, we know that $G(\boldsymbol\Lambda)$ is hyperbolic whenever $\boldsymbol\Lambda$ is finite and $C'(1/6)$. More information on this construction - graphical small cancellation theory - can be found in \cite{Gro03} where it was introduced, and in \cite{Ol06, Osaj-14-expander}.

We require one auxiliary result on the separation of hyperbolic graphs.

\begin{thm}\label{thm:hypsep} Let $X$ be a hyperbolic graph with bounded degree. There exists some $k$ such that $\sep_X(n)\preceq n^{\frac{k-1}{k}}$.
\end{thm}
\begin{proof} If $Y_k$ is a graph of bounded degree which is quasi-isometric to hyperbolic $m$-space $\mathbb{H}^m$, then $\sep_{Y_k}(n)\preceq \log(n)$ if $m=2$ and $\sep_{Y_k}(n)\preceq n^{\frac{m-2}{m-1}}$ if $m\geq 3$ \cite{BenSchTim-12-separation-graphs}.

Every hyperbolic graph $X$ of bounded degree admits a quasi-isometric embedding into some hyperbolic space $\mathbb{H}^m$ \cite{BS00} and hence a quasi-isometric embedding into some $Y_k$. Therefore $\sep_X(n)\preceq n^{\frac{k-1}{k}}$.
\end{proof}

We impose two additional conditions on family of graphs $\boldsymbol \Gamma$ which are both satisfied by taking a suitably sparse subsequence $(\Gamma^k)$.

Set $\Gamma^1=\Gamma_1$. Firstly, we ensure that for every $k$, and every subset $\boldsymbol\Lambda\subseteq \set{\Gamma^1,\dots,\Gamma^k}$ we have $\sep_{G(\boldsymbol\Lambda)}(n) < \frac{n}{k}$ whenever $n\geq \abs{\Gamma^{k+1}}$.

Secondly, for every $k$, $g(\Gamma^{k+1})>2\abs{\Gamma^k}$.

\medskip
The first is possible as for each $k$ we consider the maximal separation of finitely many hyperbolic groups, which all have sublinear separation by Theorem \ref{thm:hypsep}. For the second, we just use the fact that Osajda's construction assumes that the girth of the sequence $\Gamma_n$ is unbounded.

Now we can prove Theorem \ref{thm:manygps}.

\begin{thm} Let $A,B$ be two infinite subsets of $\N$ with $C=A\setminus B$ infinite. Define $\boldsymbol\Lambda(A),\boldsymbol\Lambda(B)= \bigsqcup_{k\in A,B}\Gamma^k$ respectively. Then there is no regular map $G(\boldsymbol\Lambda(A))\to G(\boldsymbol\Lambda(B))$.
\end{thm}
\begin{proof}
If $k\in C$ then $\Gamma^k$ is an isometrically embedded subgraph of $G(\boldsymbol\Lambda(A))$ so $\sep_{G(\boldsymbol\Lambda(A))}(\abs{\Gamma^k}) \geq \abs{\Gamma^k}\frac{\epsilon}{4}$, by Proposition $\ref{prop:upbdsep}$.

Now let $\Gamma$ be a connected subgraph of $G(\boldsymbol\Lambda(B))$ with at most $\abs{\Gamma^k}$ vertices. Using the assumption that $g(\Gamma^{k+1})>2\abs{\Gamma^k}$, we see that $\Gamma$ isometrically embeds in some $G(\boldsymbol\Lambda)$ with $\boldsymbol\Lambda\subseteq \set{\Gamma^1,\dots,\Gamma^{k-1}}$. Hence, $\sep_{G(\boldsymbol\Lambda(B))}(\abs{\Gamma^k}) < \frac{1}{k}\abs{\Gamma^k}$.

As $C$ is infinite we deduce that $\sep_{G(\boldsymbol\Lambda(A))}(n) \not\preceq \sep_{G(\boldsymbol\Lambda(B))}(n)$.
\end{proof}

\section{Finite asymptotic dimension}

In this section we give upper bounds on the separation profile of graphs with finite asymptotic dimension. This yields a quantitative version of the fact that such graphs do not contain expanders.

\begin{defin}\label{defn:asymdim} Let $X$ be a metric space. We say $X$ has \textit{asymptotic dimension} at most $m$ if there exists a function $h:\R_+\to \R_+$ such that for all $r>0$ we can partition $X$ into $m+1$ subsets $X_0,\dots,X_m$ and each $X_i$ into sets $X_{i,j}$ with $\diam(X_{i,j})\leq h(r)$ and such that $d(X_{i,j},X_{i,j'})> r$ whenever $j\neq j'$.

We say $X$ has \textit{Assouad-Nagata dimension} at most $m$ if the above holds with $h(r)\leq Cr$ for some constant $C>0$.
\end{defin}

We define the growth function of a graph $X$ to be $\gamma_X:\N\to\N\cup\set{\infty}$ where $\gamma(n)$ is the maximal cardinality of a closed ball of radius $n$ in $X$.

We now prove Theorem \ref{thm:finAN} as a consequence of the more general result stated below.

\begin{thm}\label{thm:asymdim} Let $X$ be a graph with asymptotic dimension at most $m-1$, let $h$ be a non-decreasing function provided by the above definition and let $\gamma(n)$ be the growth function of $X$. There exists a constant $k=k(m)$ such that 
\[
	\sep_X(n) \leq \frac{kn}{f_h\left(\frac{n}{2m}\right)},
\] 
where we define $f_h(n)=\max\setcon{k}{\gamma(h(k))\leq n}$.
\end{thm}

\begin{proof} Let $\Gamma$ be a subgraph of $X$ containing $n$ vertices, then setting $r=f_h\left(\frac{n}{2m}\right)$, we obtain a cover of $X$ by $m$ subsets $B_0,\dots,B_{m-1}$ such that each $B_i$ decomposes into subsets $B_{i,j}$ of diameter at most $h(r)$ which are $r$ disjoint.

It follows immediately that for some $i$, $\Gamma\cap B_i$ contains at least $\frac{n}{m}$ vertices, without loss we assume this is true for $i=0$.
Now, each $B_{i,j}$ contains at most $\frac{n}{2m}$ vertices in $\Gamma$, so $V(\Gamma)$ meets at least two such $B_{0,j}$.

Let $U$ be a union of sets $U_j=B_{0,j}\cap \Gamma$ with between $\frac{n}{4m}$ and $\frac{n}{2m}$ vertices - a greedy search will achieve this. Notice that the complement, $V$, of the $r$ neighbourhood of $U$ contains at least $n/2m$ vertices of $B_0$.

Consider the sets $C_l = \setcon{v\in V\Gamma}{d_X(v,U)=l}$ with $1\leq l\leq r$.

It is clear that for each $l$, $U$ and $V$ are disjoint from $C_l$ and cannot be connected by a path in $\Gamma\setminus C_l$. There exists an $l$ such that $\abs{C_l}\leq \frac{n}{r}$ by the pigeon-hole principle. Fix such an $l$.

Any largest connected component $\Gamma_1$ of $\Gamma\setminus C_l$ is disjoint from either $U$, $V$ or both, so it contains at most $(1-\frac{1}{4m})n$ vertices. If $\abs{\Gamma_1}\leq \frac{n}{2}$ then we are done.

If not, then applying the above procedure to $\Gamma_1$ we can remove another set $C'_{l'}$ of at most $\frac{n}{r}$ vertices; a largest connected component $\Gamma_2$ of $\Gamma_1\setminus C'_{l'}$ contains at most $n(1-\frac{1}{4m})^2$ vertices.

Repeating this process $k$ times, where $k$ satisfies $(1-\frac{1}{4m})^k\leq 1/2$, we deduce that after removing a set $C$ containing at most $kn/r$ vertices of $\Gamma$, the maximal connected component of $\Gamma\setminus C$ is contained in $\Gamma_k$ which has at most $n(1-\frac{1}{4m})^k \leq \frac{n}{2}$ vertices.
\end{proof}

To complete the proof of Theorem \ref{thm:finAN} we notice that $f_h(n)\succeq\log(n)$ for any bounded valence graph with finite Assouad-Nagata dimension. More specifically we recover the upper bound $\sep_X(n)\preceq n^{\frac{d-1}{d}}$ for any group with polynomial growth of degree at most $d$. Every vertex transitive graph with polynomial growth is quasi-isometric to the Cayley graph of a finitely generated group \cite{Sab,Tro}, so the above result naturally extends to all vertex transitive graphs of polynomial growth.

\begin{prob} Is it true that every vertex transitive graph $X$ with polynomial growth of degree $d$ satisfies $\sep_X(n)\succeq n^{\frac{d-1}{d}}$?
\end{prob}

\end{document}